\documentclass[11pt]{amsart}
\usepackage{amsmath, amsthm, amsfonts, amscd, amssymb, eucal, latexsym, mathrsfs, stmaryrd, bm, microtype}
\usepackage[all]{xy}
\usepackage{mathtools}
\usepackage[shortlabels]{enumitem}
\usepackage{comment}
\usepackage{tikz-cd}
\usepackage{tikz}
\usepackage[dvipsnames]{xcolor}
\usepackage{hyperref}
\usepackage{caption}

\newtheorem*{theorem*}{Theorem}
\newtheorem{theorem}{Theorem}[section]

\newtheorem{lemma}[theorem]{Lemma}

\theoremstyle{definition}
\newtheorem{definition}[theorem]{Definition}
\newtheorem{remark}[theorem]{Remark}

\newtheorem*{question*}{Question}
\newtheorem{problem}[theorem]{Problem}

\theoremstyle{plain}
\newcounter{thmintronumber}
\newtheorem{theoreminum}[thmintronumber]{Theorem}
\theoremstyle{definition}

\theoremstyle{plain}
\newcounter{theoremintro}
\newcounter{conjectureintro}
\newcounter{corollaryintro}
\newtheorem{theoremi}[theoremintro]{Theorem}

\newcommand{\Zb}{{\mathbb Z}}
\newcommand{\Nb}{{\mathbb N}}
\newcommand{\Rb}{{\mathbb R}}

\newcommand{\abs}[1]{\left|#1\right|}
\DeclarePairedDelimiter\ceil{\lceil}{\rceil}

\newcommand{\act}[1][]{\overset{#1}{\curvearrowright}}
\renewcommand{\phi}{\varphi}

\newcommand{\brak}[1]{\left(#1\right)}
\DeclareMathOperator{\id}{id}
\DeclareMathOperator{\Mod}{mod}

\title{Quantitative orbit equivalence for $\Zb$-odometers}
\author{Petr Naryshkin}
\address{Petr Naryshkin,
Institute for Advanced Study, 1 Einstein Dr, Princeton, NJ 08540, USA}
\email{penaryshkin@ias.edu}

\author{Spyridon Petrakos}
\address{Spyridon Petrakos,
Department of Mathematical Sciences, Chalmers University of Technology and University of Gothenburg, SE-412 96 Gothenburg, Sweden}
\email{petrakos@chalmers.se}
\date{\today}

\begin{document}

\begin{abstract}
We prove that any two $\Zb$-odometers are sub-$L^1$-orbit equivalent, greatly strengthening previous results and giving a definitive picture of quantitative orbit equivalence for these systems.
\end{abstract}

\maketitle

\section{Introduction}

Orbit equivalence between two probability measure preserving (p.m.p.) actions of countable groups is a natural notion that has been studied extensively in ergodic theory. When the acting groups are non-amenable, this relation is highly non-trivial, and studying it has led to a plethora of rigidity results, starting with the work of Zimmer in the early 1980s \cite{Zim80} and flourishing with the works of Furman \cite{Fur99a,Fur99b}, the development of Popa's deformation/rigidity theory \cite{Pop07}, Monod and Shalom's work on bounded cohomology \cite{MS06}, and works of many authors along these lines---not to mention parallel developments in the closely related theory of measure equivalence (see surveys \cite{Fur09,Gab11} for an overview of early results). 

This comes in stark contrast with the amenable side, where remarkable results of Dye \cite{Dye59,Dye63}, and Ornstein and Weiss \cite{OW80}, guarantee a total collapse of the theory: any two ergodic p.m.p. actions of countably infinite amenable groups are orbit equivalent. In addition, this (common) orbit equivalence class exactly remembers the amenability of the acting group in the free case \cite{Zim78}.

The situation becomes more interesting if one imposes additional conditions on the resulting cocycles, a theory broadly referred to as \emph{quantitative orbit equivalence}. Different degrees of rigidity can then emerge, as already witnessed in the following---by now classical---theorem of Belinskaya \cite{Bel68}.

\begin{theoreminum}\label{thm: Belinskaya}
If two ergodic p.m.p. transformations $T$ and $S$ on a standard probability space are integrably orbit equivalent, then $T$ and $S$ are flip-conjugate (i.e., either $T\cong S$ or $T\cong S^{-1}$). 
\end{theoreminum}

Integrable orbit equivalence is only one instance ($p=1$) of $L^p$-orbit equivalence, which itself is very prominent in the theory. Systematic interest in quantitative orbit equivalence is justified by more recent results showing that with appropriate conditions (not necessarily in terms of integrability; e.g., Shannon orbit equivalence) certain invariants must be preserved, including the entropy of the action \cite{Aus16b,KL21,KL24,BL24} and the F{\o}lner profile of the acting group (in particular, its growth) \cite{Aus16a,DKLMT22}. 

Despite the above, relatively little is known about concrete systems in terms of their various orbit equivalence classes\footnote{The term here is used loosely; contrary to what the name suggests, it is often unclear whether quantitative orbit equivalence of a given kind is transitive.\label{footnote}}. Even in seemingly tractable cases, only partial results are available (see Section~\ref{sec: prelim} for definitions and notation):

\begin{theoreminum}\label{thm: prev results}
Let $\Zb \act X$ be a free ergodic  p.m.p. action generated by $T \colon X \to X$ and let $\omega\colon\Zb_{\geq0}\to\Rb_{\geq0}$ be a sublinear function.
\begin{itemize}
    \item If there is some $n \ge 2$ such that $T^n$ is ergodic, then there exists an action $\Zb \act Y$ that is $\omega$-orbit equivalent to $\Zb \act X$ but not flip-conjugate to it \cite{CJLMT23}.
    \item If $\Zb \act X$ is any odometer, then there exists an action $\Zb \act Y$ that is $\omega$-orbit equivalent to $\Zb \act X$ but not flip-conjugate to it \cite{Cor25b}.
    \item Every $\Zb$-odometer is $L^{<1/3}$-orbit equivalent to the universal odometer \cite{KL24,Cor25a}.
\end{itemize}
\end{theoreminum}

In particular, there has been no satisfactory answer to the following question, even for special classes of systems.

\begin{question*}\label{ques: main}
Given two systems $\Zb \act X$ and $\Zb \act Y$, what is the best integrability condition that an orbit equivalence between them can satisfy?
\end{question*}

The main result of this paper provides a complete answer to this question for odometers.

\begin{theoremi}[Theorem \ref{thm:main with omega}]\label{thm:main}
Any two $\Zb$-odometers are sub-$L^1$-orbit equivalent.
\end{theoremi}

In other words, within the class of odometers, integrable orbit equivalence is the same as conjugacy (which coincides here with flip-conjugacy), and any weaker integrability condition leads to a trivial relation. This generalizes the results of Kerr-Li and Correia in Theorem~\ref{thm: prev results}.

The strategy of proof for these results was to \emph{construct} a transformation $S \colon X \to X$ that lies in the full group of $T$, induces the same orbits as $T$, and such that the cocycle satisfies the prescribed integrability condition. The downside of this approach is that it is difficult to control what the system generated by $S$ is isomorphic to. That is why these results are, in a sense, one-directional: a given odometer is orbit equivalent either with \emph{some} other system or a very special odometer, the universal one. 

The approach that we take is loosely inspired by Keane and Smorodinsky's proof of Ornstein's theorem \cite{KS79}, and is specifically designed to take into account \emph{both} systems. The idea is to inductively construct back and forth maps between finite factors of the odometers such that the resulting diagram (almost) commutes (see Figure~\ref{fig: back and forth maps}). If these maps satisfy certain quantitative estimates, one can then take a limit and obtain an orbit equivalence between $\Zb \act X$ and $\Zb \act Y$ that satisfies the desired integrability condition. 

Although in this paper we only address the case of odometers (as those are precisely the systems that can be approximated by finite actions in a very strong sense), we believe that this framework is fairly flexible and could be used to obtain further results in this direction.

\subsection*{Acknowledgements}The authors are grateful to David Kerr, Corentin Correia, and François Le Maître for helpful comments and discussions.

The first named author was supported by the Dynasnet European Research Council Synergy project -- grant number ERC-2018-SYG 810115, as well as the Marvin V. and Beverly J. Mielke Endowed Fund through the Institute of Advanced Study. 

The second named author was supported by the Deutsche Forschungsgemeinschaft (DFG, German Research Foundation) under Germany’s Excellence Strategy EXC 2044-390685587, Mathematics M\"{u}nster: Dynamics–Geometry–Structure, by the SFB 1442 of the DFG, and by the Knut and Alice Wallenberg Foundation through a postdoc grant.

\section{Preliminaries}\label{sec: prelim}

\subsection{Quantitative orbit equivalence}\label{ssec: QOE}

Let $G, H$ be countable discrete groups and let $G \act X$ and $H \act Y$ be (essentially) free p.m.p. actions. A measure isomorphism $\phi \colon X \to Y$ is called an \emph{orbit equivalence} from $G$ to $H$\footnote{There is no actual directionality in this definition. Inserting it artificially will be useful later.} if 
\[
\phi(Gx) = H\phi(x)
\]
for almost every $x \in X$. 

To each such $\phi$ we associate the \emph{cocycles} 
\[
\lambda_\phi \colon G \times X \to H \quad \mbox{and} \quad \kappa_\phi \colon H \times Y \to G
\]
defined by
\[
\phi(gx) = \lambda_\phi(g, x)\phi(x) \quad \mbox{and} \quad \phi^{-1}(hy) = \kappa_\phi(h, y)\phi^{-1}(y)
\]
for every $g \in G$ and $h \in H$, and almost every $x \in X$ and $y \in Y$.

Assume now that $G$ and $H$ are finitely generated and equip them with the word-length metrics $\abs{\cdot}_G \colon G \to \Zb_{\ge 0}$ and $\abs{\cdot}_H \colon H \to \Zb_{\ge 0}$ corresponding to some choice of finite generating sets. Let $\omega_i \colon \Zb_{\ge 0} \to \Rb_{\ge 0},i\in\{1,2\}$ be non-decreasing functions.

\begin{definition}
We say that two actions $G \act X$ and $H \act Y$ are \emph{$(\omega_1,\omega_2)$-orbit equivalent} if there exists an orbit equivalence $\phi \colon X \to Y$ from $G$ to $H$ such that for every $g \in G$ and every $h \in H$, there exist constants $c_g, c_h>0$ with
\begin{equation}\label{eq: omega-OE def}
\int_X \omega_1\brak{\frac{\abs{\lambda_\phi(g, x)}_H}{c_g}}dx < \infty \quad \mbox{and} \quad \int_Y \omega_2\brak{\frac{\abs{\kappa_\phi(h, y)}_G}{c_h}}dy < \infty.
\end{equation}
When $\omega_1=\omega_2=\omega$, we simply call them \emph{$\omega$-orbit equivalent}. In particular, we say that the actions are
\begin{itemize}
    \item \emph{integrably orbit equivalent} if they are $\omega$-orbit equivalent for $\omega(n) = n$,
    \item \emph{$L^p$-orbit equivalent}, where $0 < p < \infty$, if they are $\omega$-orbit equivalent for $\omega(n) = n^p$,
    \item \emph{$L^{<p}$-orbit equivalent}, if they are $L^q$-orbit equivalent for all $0<q<p$.
    \item \emph{sub-$L^p$-orbit equivalent}, if they are $\omega$-orbit equivalent for all $\omega$ with \[
    \lim_n\frac{\omega(n)}{n^p}=0.
    \]
    When $p=1$, such an $\omega$ is called \emph{sublinear}.
\end{itemize}
\end{definition}

Note that $L^p$-orbit equivalence is stronger than $L^q$-orbit equivalence for $q < p$, and
sub-$L^p$-orbit equivalence is stronger than $L^{<p}$-orbit equivalence.

\begin{remark}\label{rem: Belinskaya optimal}
In terms of the above, Theorem~\ref{thm: Belinskaya} implies that if two $\Zb$-systems are not flip-conjugate, then they are at most sub-$L^{1}$-orbit equivalent.
\end{remark}

\begin{lemma}[{see \cite[Proposition~2.22]{DKLMT22}}]
Let $G \act X$ and $H \act Y$ be as above, and let $S \subseteq G$ and $R \subseteq H$ be finite generating sets. Then, to check that $\phi \colon X \to Y$ is an ($\omega_1,\omega_2$)-orbit equivalence, it is sufficient to check that \eqref{eq: omega-OE def} holds for all $g \in S$ and $h \in R$.
\end{lemma}

In particular, since we are interested in odometers, we will always work with the canonical generator. For the sake of brevity, we will also use the following.

\begin{definition}
Let $\omega\colon\Zb_{\geq0}\to\Rb_{\geq0}$ be non-decreasing, and let $\lambda\colon X\to\mathbb{Z}$ be a map. We define the \emph{$\omega$-norm} of $\lambda$ to be the integral
\[
    \int_X \omega(|\lambda(x)|)dx.
\]
\end{definition}

\subsection{\texorpdfstring{$\Zb$}{Z}-odometers}\label{ssec: odometers}
Let $k_n \in \Nb$ be an increasing sequence of natural numbers such that $k_n | k_{n+1}$ (that is, $k_n$ divides $k_{n+1}$). Consider the finite dynamical systems $\Zb \act \Zb/(k_n\Zb)$ (where the set $\Zb/(k_n\Zb)$ is equipped with the uniform probability measure) and note that the maps in Figure~\ref{fig: odometer maps}
\begin{figure}[ht]
\centering
\begin{tikzcd}[column sep=3cm]
\ldots&\arrow[l, "\Mod(k_{n-1})" description] \Zb/k_n\Zb & \arrow[l, "\Mod(k_n)" description] \Zb/k_{n+1}\Zb & \arrow[l, "\Mod(k_{n+1})" description]\ldots
\end{tikzcd}
\caption{Inverse system defining an odometer.}\label{fig: odometer maps}
\end{figure}
are both measure-preserving and $\Zb$-equivariant. The \emph{odometer} associated with the sequence $(k_n)_n$ is the inverse limit of the systems in Figure~\ref{fig: odometer maps}.

Now, let $q$ be a prime number and define 
\[
r_{q, n} = \max(r \in \Nb : q^r | k_n) \quad \mbox{and} \quad r_q = \lim_{n \to \infty}r_{q, n} \in \Nb \cup \{\infty\}.
\]
The sequence $r_q \in \Nb \cup \{\infty\}$ defines the \emph{supernatural number} associated with the sequence $(k_n)$, typically written in multiplicative form:
\[
2^{r_2}3^{r_3}5^{r_5}\ldots.
\]
Two sequences $(k_n)$ and $(k^\prime_n)$ produce isomorphic odometers if and only if their associated supernatural numbers are equal. See \cite{Dow05}.

\section{The main result}\label{sec: main proof}

Unless otherwise specified, finite sets are equipped with the uniform probability measure.

\begin{definition}\label{def: cocycle non-free}
We use the notation $[m]$ for the interval $\{0, 1, \ldots, m-1\}$ in $\Zb$. Note that it is the canonical set of representatives of the quotient $\Zb/m\Zb$ and hence can be identified with the latter. 

Let $\phi \colon [m] \to \Zb$ be an arbitrary map. We define its \emph{(generating) cocycle} to be the map $\lambda_\phi \colon [m] \to \Zb$ given by
\[
\lambda_\phi(x) = \phi((x+1) \Mod m) - \phi(x). 
\]
\end{definition}

\begin{lemma}\label{lem: constructing psi}
Set $k_{-1}=1$. Let $(k_n)_{n=0}^\infty$ be an increasing sequence of natural numbers such that $k_0=1$, $k_n | k_{n+2}$ for every $n \in \Nb$, and $k_{n+1} > k_{n-1}k_{n}$ for every $n \in \Nb$. Let $\omega \colon \Zb_{\ge 0} \to \Rb_{\ge 0}$ be a non-decreasing function. 

Then there exists a sequence of bijections
\[
\psi_n \colon [k_{n-1}k_{n}] \to [k_{n-1}k_{n}]
\]
that satisfy the following properties:
\begin{enumerate}
\item The diagram in Figure~\ref{fig: diagram for psi_n} is commutative up to a set of measure $k_{n-2}k_{n-1}/k_n$. More precisely, the set of elements $x \in [k_{n-1}k_n]$ such that
\[
\brak{\psi_{n-1}^{-1} \circ \Mod (k_{n-2}k_{n-1}) \circ \Mod (k_n)}(x) = \brak{\Mod(k_{n-2}k_{n-1}) \circ \psi_{n}} (x)
\]
has measure at least $1 - k_{n-2}k_{n-1}/k_n$.
\begin{figure}[ht]
\centering
\begin{tikzcd}[column sep=1.8cm,row sep=1cm]
& {[k_{n-1}k_n]}\arrow[rdd,"\psi_n" description]\arrow[ld,"\Mod(k_n)" description] & \\
{[k_n]}\arrow[dd,"\Mod(k_{n-2}k_{n-1})" description]& & \\
& & {[k_{n-1}k_n]}\arrow[ldd,"\Mod(k_{n-2}k_{n-1})" description] \\
{[k_{n-2}k_{n-1}]}\arrow[rd,"\psi_{n-1}^{-1}" description] & & \\
& {[k_{n-2}k_{n-1}]} &
\end{tikzcd}
\caption{Compatibility condition for $\psi_n$.}\label{fig: diagram for psi_n}
\end{figure}
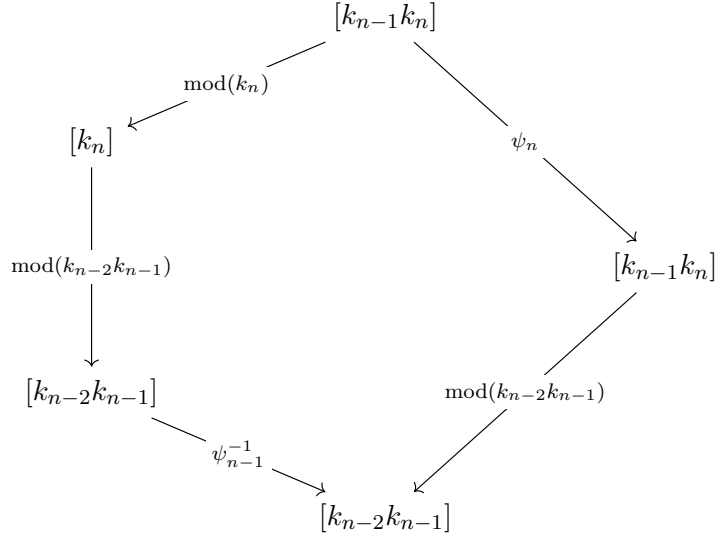

\item $\psi_n(0) = 0$ and $\psi_n(k_{n-1}k_n-1) = k_{n-1}k_n-1$.

\item The $\omega$-norms of the generating cocycles for $\psi_n$, $\psi_n^{-1}$, and $\Mod(k_n)\circ \psi_n^{-1}$ are at most
\[
\sum_{m=1}^{n}\left[\frac{2}{k_{m}}\omega(k_{m-1}k_{m})+\brak{\frac{1}{k_{m-2}k_{m-1}}+\frac{k_{m-2}k_{m-1}}{k_m}}\omega(1)\right].
\]
\end{enumerate}
\end{lemma}

\begin{proof}
We construct maps $\psi_n$ inductively, starting with $\psi_1 = \id$.

Assume $\psi_n$ is defined and let $x \in [k_nk_{n+1}]$. Let $x = ak_{n+1} + b$, let $k_{n+1} = ck_{n-1}k_n + d$, and let $b = ek_{n-1}k_n + f$ be Euclidean divisions. We define (see Figure~\ref{fig: definition of psi_n})
\[
\psi_{n+1}(x) = 
\begin{cases}
(ac+e)k_{n-1}k_n + \psi_n^{-1}(f), & e < c, \\
ck_nk_{n-1}k_n + ad + f, & e=c.
\end{cases}
\]
\begin{figure}[ht]
\centering
\begin{tikzpicture}[scale=0.3]
    \path (-1.5,0) (51.5,0);
	\draw[<->] (0,8) -- (25,8) node[above] {$k_nk_{n+1}$} -- (50,8);
	\draw[<->] (0,5) -- (6.5,5) node[above] {$k_{n+1}$} -- (13,5);
	\draw[<->] (0,2) -- (1,2) node[above] {$k_{n-1}k_n$} -- (2,2);
	\draw[<->] (12,2) -- (12.5,2) node[above] {$d$} -- (13,2);
	\draw[Green] (0,0) -- (6.8,0);
	\draw[Green] (9.2,0) -- (12,0);
    \draw[red] (12,0) -- (13,0);
    \draw[Green] (13,0) -- (19.8,0);
	\draw[Green] (22.2,0) -- (25,0);
    \draw[red] (25,0) -- (26,0);
    \draw[Green] (26,0) -- (29.5,0);
	\draw[Green] (33.5,0) -- (36,0);
    \draw[red] (36,0) -- (37,0);
    \draw[Green] (37,0) -- (43.8,0);
	\draw[Green] (46.2,0) -- (49,0);
    \draw[red] (49,0) -- (50,0);
	\foreach \x in {8, 21, 31.5, 45}{
		\node at (\x,0) {$\ldots$};
	}
	\draw (0,-1.5) -- (0,1.5);
	\draw (50,-1.5) -- (50,1.5);
	\foreach \x in {13,26,37}{
		\draw (\x,-1) -- (\x,1);
	}
	\foreach \y in {0,13,37}{
		\foreach \x in {2,4,6,10,12}{
			\draw (\y+\x,-0.5) -- (\y+\x,0.5);
		}
	}
	\draw[Green] (0,-10) -- (6.8,-10);
	\draw[Green] (9.2,-10) -- (18.8,-10);
	\draw[Green] (21.2,-10) -- (25,-10);
	\draw[Green] (29,-10) -- (36.8,-10);
	\draw[Green] (39.2,-10) -- (42,-10);
    \draw[red] (42,-10) -- (44.5,-10);
	\draw[red] (48.5,-10) -- (50,-10);
	\foreach \x in {8,20,27,38,46.5}{
		\node at (\x,-10) {$\ldots$};
	}
	\draw (0,-11.5) -- (0,-8.5);
	\draw (50,-11.5) -- (50,-8.5);
	\foreach \y in {0,12,30}\foreach \x in {0,2,4,6,10,12}{
		\draw (\y+\x,-10.5) -- (\y+\x,-9.5);
	}
	\draw (42,-11) -- (42,-9);
	\foreach \x in {43,44,49}{
		\draw (\x,-10.5) -- (\x,-9.5);
	}
	\foreach \x in {0,1}{
		\foreach \y in {1,3,5,11}{
			\draw[Green,->] (13*\x+\y,-0.5) .. controls (13*\x+\y,-5) and (12*\x+\y,-5) .. (12*\x+\y,-9.5);
		}
	}
	\foreach \x in {1,3,5,11}{
		\draw[Green,->] (37+\x,-0.5) .. controls (37+\x,-5) and (30+\x,-5) .. (30+\x,-9.5);
	}
	\draw[red,->] (12.5,-0.5) .. controls (12.5,-5) and (42.5,-5) .. (42.5,-9.5);
	\draw[red,->] (25.5,-0.5) .. controls (25.5,-5) and (43.5,-5) .. (43.5,-9.5);
	\draw[red,->] (49.5,-0.5) -- (49.5,-9.5);
	\draw[Green,->] (0,-13) -- (3,-13) node[right,black] {$:\psi_n^{-1}$};
	\draw[red,->] (0,-15.5) -- (3,-15.5) node[right,black] {$:\id_{[d]}$};
\end{tikzpicture}
\caption{Definition of $\psi_{n+1}$.}\label{fig: definition of psi_n}
\end{figure}
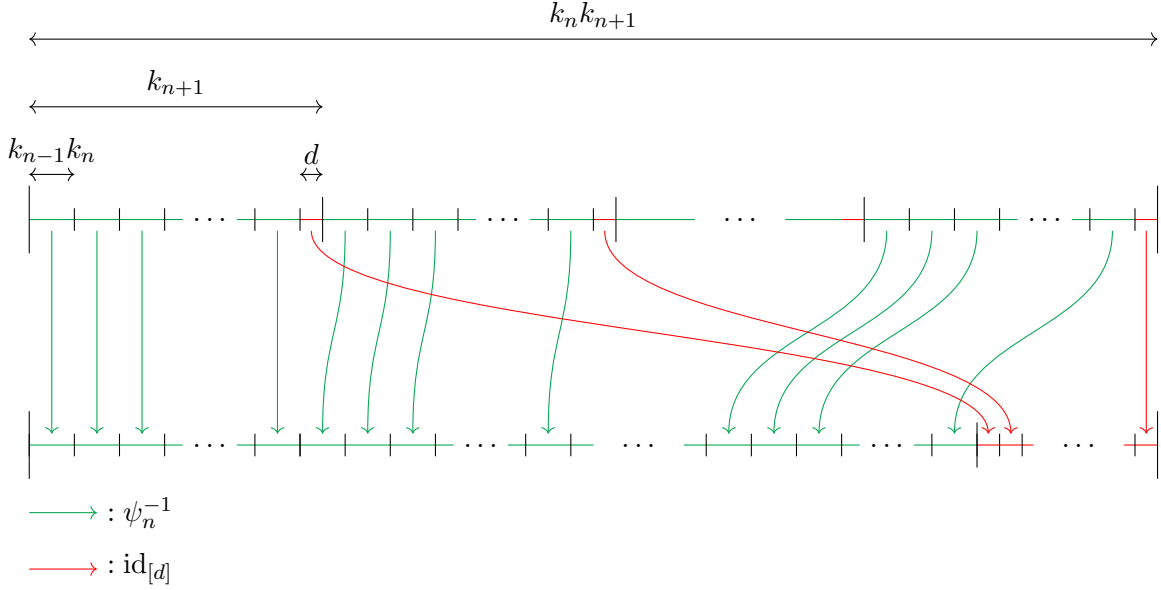

It remains to prove that properties $(1)-(3)$ are satisfied. Indeed, one can easily check that if $e < c$ (that is, if $x$ is in the green part of Figure~\ref{fig: definition of psi_n}) then the diagram in Figure~\ref{fig: diagram for psi_n} commutes. The measure of such points is equal to
\[
1 - \frac{d}{k_{n+1}} \ge 1 - \frac{k_{n-2}k_{n-1}}{k_n},
\]
which proves property $(1)$. Property $(2)$ is immediate from the definition.

To check property $(3)$, we first note that the cocycle for $\psi_{n+1}$ can be calculated explicitly in the following three cases (that is, when $x$ and $x+1$ have the same color in the top half of Figure~\ref{fig: definition of psi_n}):
\[
\abs{\lambda_{\psi_{n+1}}(x)} = 
\begin{cases}
\abs{\lambda_{\psi^{-1}_n}(f)}, & e<c\text{ and }f< k_{n-1}k_n-1, \\
1, & e < c-1\text{ and }f= k_{n-1}k_n-1, \\
1, & e = c\text{ and }f < d-1. \\
\end{cases}
\]
The first case is used to invoke the inductive hypothesis, while the second and third contribute the
\[
    \frac{1}{k_{n-1}k_n}\omega(1)+\frac{k_{n-1}k_n}{k_{n+1}}\omega(1)
\]
summand.

That leaves $2k_n$ points $x$ such that $x$ and $x+1$ have different colors in Figure~\ref{fig: definition of psi_n}. We denote this set of points by $E$. For $x \in E$ simply estimate $\abs{\lambda_{\psi_{n+1}}(x)}$ by the length of the interval, which contributes the
\[
\frac{2}{k_{n+1}}\omega(k_nk_{n+1})
\]
summand, finishing the argument for $\psi_{n+1}$.

It is a matter of straightforward calculation for one to check that the same $\omega$-norm estimate holds for the cocycle of $\psi_{n+1}^{-1}$. Note finally that 
\[
\lambda_{\psi_{n+1}^{-1}}(y) = \lambda_{\Mod (k_{n+1}) \circ \psi_{n+1}^{-1}}(y),
\]
unless $\psi_{n+1}^{-1}(y) \in E$. The latter case corresponds exactly to the points for which we used the whole length of the interval $[k_nk_{n+1}]$ as a very crude estimate. Therefore, since the range interval of $\Mod (k_{n+1})\circ \psi_{n+1}^{-1}$ is even smaller, the same crude estimate certainly holds.

\end{proof}

\begin{definition}
We define $\phi_n \colon [k_{n+1}] \to [k_n]$ to be the composition
\[
\phi_n = \Mod (k_n)\circ \psi_n^{-1} \circ \Mod (k_{n-1}k_n).
\]
\end{definition}

\begin{lemma}\label{lem: prop of phi}
The maps $(\phi_n)_n$ satisfy the following properties:
\begin{enumerate}[(i)]
\item $\phi_n \circ \phi_{n+1} = \Mod (k_n) \colon [k_{n+2}] \to [k_n]$ up to a set of measure at most 
\[
\frac{k_{n-1}k_n}{k_{n+1}} + \frac{k_nk_{n+1}}{k_{n+2}}.
\]
\item For all sets $A \subseteq [k_n]$ we have that 
\[
\mu_{n+1}(\phi_n^{-1}(A)) \le \brak{1 + \frac{k_{n-1}k_n}{k_{n+1}}}\mu_n(A),
\]
where $\mu_n$ and $\mu_{n+1}$ are the uniform probability measures on $[k_{n+1}]$ and $[k_n]$ respectively.
\item $\lambda_{\phi_n}(x)=\lambda_{\Mod (k_n) \circ \psi_n^{-1}}(x\Mod(k_{n-1}k_n))$ and therefore the $\omega$-norm of $\lambda_{\phi_n}$ is at most
\[
\sum_{m=1}^{n}\left[\frac{2}{k_{m}}\omega(k_{m-1}k_{m})+\brak{\frac{1}{k_{m-2}k_{m-1}}+\frac{k_{m-2}k_{m-1}}{k_m}}\omega(1)\right] + \frac{k_{n-1}k_n}{k_{n+1}}\omega(k_n).
\]
\end{enumerate}
\end{lemma}

\begin{proof}
Property (i) is straightforward to check using condition $(1)$ of Lemma~\ref{lem: constructing psi}, while (ii) and the fact that 
\[
\lambda_{\phi_n}(x)=\lambda_{\Mod (k_n) \circ \psi_n^{-1}}(x\Mod(k_{n-1}k_n))
\]
follow from the definition and noticing that $\phi_n$ is at most $k_{n-1}\ceil{k_{n+1}/(k_nk_{n-1})}$-to-$1$. Finally, the estimate for the $\omega$-norm of $\lambda_{\phi_n}$ is clear from property $(3)$ of Lemma~\ref{lem: constructing psi} and another crude estimate for the last $k_{n+1}\Mod(k_{n-1}k_n)$ points of the interval $[k_{n+1}]$.
\end{proof}

\begin{theorem}\label{thm:main with omega}
Let $\Zb \act X$ and $\Zb \act Y$ be two odometers.  Let $\omega \colon \Zb_{\ge 0} \to \Rb_{\ge 0}$ be a non-decreasing sublinear function. Then $\Zb \act X$ and $\Zb \act Y$ are $\omega$-orbit equivalent. In fact, the $\omega$-norm of the cocycle can be made to be less than $\omega(1)+\delta$ for any $\delta > 0$.
\end{theorem}

\begin{proof}
Let the odometers $X$ and $Y$ be given. Inductively choose integers
\[
1<k_1< k_2< \ldots<\allowbreak k_n< \ldots
\]
such that $k_n | k_{n+2}$ and
\begin{enumerate}[I.]
    \item
    \[
    \varprojlim(\Zb/k_{2n+1}\Zb)= X,\ \varprojlim(\Zb/k_{2n}\Zb) = Y,
    \]
    \item 
    \[
    \sum_{n=0}^\infty \frac{k_{n-1}k_n}{k_{n+1}}\omega(k_n) < \frac{\delta}{3},
    \]
    \item 
    \[
    \sum_{n=1}^\infty \frac{2}{k_n}\omega(k_{n-1}k_n) < \frac{\delta}{3}.
    \]
\end{enumerate}

Let $\phi_n$ be the functions constructed from this sequence. Properties (i) and (ii) of Lemma~\ref{lem: prop of phi} imply that there are almost everywhere defined limit functions (see Figure~\ref{fig: back and forth maps})
\[
    \phi_e = \varprojlim \phi_{2n} \colon X \to Y,\ \phi_o = \varprojlim \phi_{2n+1} \colon Y \to X,
\]
both of which are measure non-increasing and such that
\[
\phi_o \circ \phi_e = id_X,\ \phi_e \circ \phi_o = id_Y.
\]
This implies that both of them are, in fact, measure isomorphisms between $X$ and $Y$. 

\begin{figure}[ht]
\centering
\begin{tikzcd}[column sep=2cm,row sep=0.5cm]
\Zb/k_1\Zb &\\
& \arrow[ul,"\phi_1" description] \Zb/k_2\Zb\\
\Zb/k_3\Zb \arrow[uu, "\Mod(k_1)" description]\arrow[ur, "\phi_2" description]&\\
& \Zb/k_4\Zb \arrow[ul, "\phi_3" description] \arrow[uu,"\Mod(k_2)" description]\\
\Zb/k_5\Zb \arrow[uu, "\Mod(k_3)" description]\arrow[ur,"\phi_4" description] & \\
\vdots\arrow[u] &\vdots\arrow[uu,"\Mod(k_4)" description]\\
X\arrow[u]\arrow[r,"\phi_e" below,shift right]& Y\arrow[u]\arrow[l,"\phi_o" above,shift right]
\end{tikzcd}
\caption{Almost commutative diagram defining $\phi_o$ and $\phi_e$.}\label{fig: back and forth maps}
\end{figure}
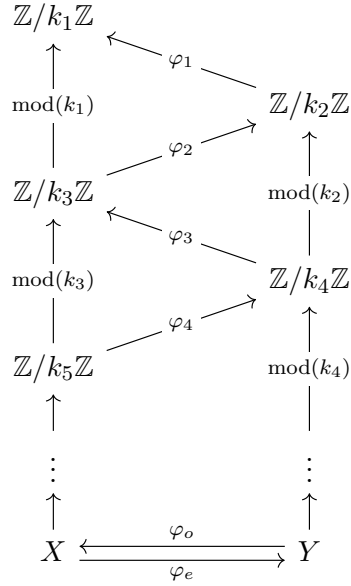

The cocycle for $\phi_e$ (resp. $\phi_o$) is a.e. the pointwise limit of the cocycles for $\phi_{2n}\circ \pi_{2n}$ (resp. $\phi_{2n+1}\circ \pi_{2n+1}$), where $\pi_n$ is the projection onto the corresponding finite factor. Hence, by Borel-Cantelli and property (iii) of Lemma~\ref{lem: prop of phi}, its $\omega$-norm is at most
\[
\sum_{m=1}^{\infty}\left[\frac{2}{k_{m}}\omega(k_{m-1}k_{m})+\frac{1}{k_{m-2}k_{m-1}}\omega(1)+\frac{k_{m-2}k_{m-1}}{k_m}\omega(1)\right]<\omega(1)+\delta,
\]
which finishes the proof.
\end{proof}

We have shown that any two $\Zb$-odometers are, in view of Remark~\ref{rem: Belinskaya optimal}, orbit equivalent in the strongest sense possible.

On the other hand, it is known that any free action of $\Zb^n$ and any free action of $\Zb^m$ are at best sub-$(L^{m/n},L^{n/m})$-orbit equivalent (defined in the obvious way), and that the corresponding \emph{dyadic} odometers are $(L^{<m/n},L^{<n/m})$-orbit equivalent \cite{DKLMT22,Cor25c}. It is therefore natural to ask:

\begin{problem}\label{prob: Z^n}
Let $\Zb^n \act X$ and $\Zb^m \act Y$ be odometers of $\Zb^n$ and $\Zb^m$, respectively. Are they necessarily sub-($L^{m/n},L^{n/m})$-orbit equivalent?
\end{problem}

\end{document}